\documentclass{amsart}

 \newtheorem{thm}{Theorem}
 
 \newtheorem{lem}[thm]{Lemma}
 \newtheorem{prop}[thm]{Proposition}
 \theoremstyle{definition}
 \newtheorem{defn}[thm]{Definition}
 \theoremstyle{remark}
 \newtheorem{rem}[thm]{Remark}
 \newcommand{\OO}{{\mathcal O}}
 \newcommand{\M}{{\mathcal M}}
\newcommand{\Nc}{{\mathcal N}}
\DeclareMathOperator{\Spec}{{Spec}}
\newcommand{\pp}{{\mathfrak p}}
\newcommand{\wh}{\widehat}
\DeclareMathOperator{\limi}{{lim}}
\newcommand{\plim}[1]{\,\underset{#1}{\underset{\leftarrow}{\limi}}\,}
\newcommand{\punto}{{\displaystyle \cdot}}
\newcommand{\oplusa}[2][]{\underset{#2}{\overset{#1}{\oplus}}}
\newcommand{\enumera}{\begin{enumerate}}
\newcommand{\eenumera}{\end{enumerate}}

\DeclareMathOperator{\Ker}{{Ker}} \DeclareMathOperator{\Ima}{{Im}}

\begin{document}

\title[A direct proof of the Theorem on Formal Functions]
 {A direct proof of the Theorem on Formal Functions}

\author{ Fernando Sancho de Salas and Pedro Sancho de Salas}
\address{Departamento de Matem\'{a}ticas, Universidad de Salamanca,
Plaza de la Merced 1-4, 37008 Salamanca, Spain}

\email{fsancho@usal.es}

\address{Departamento de Matem\'{a}ticas, Universidad de Extremadura,
Avenida de Elvas s/n, 06071 Badajoz, Spain}

 \email{sancho@unex.es}

\thanks {The firs author was supported by research projects MTM2006-04779 (MEC)
and SA001A07 (JCYL)}




\date{\today}

\dedicatory{}



\begin{abstract}
We give a direct and elementary proof of the theorem on formal
functions by studying the behaviour of the Godement resolution of
a sheaf of modules under completion.

\end{abstract}

\maketitle

\section*{Introduction}

Let $\pi\colon X\to \Spec A$ be a proper scheme over a ring $A$.
Let $\M$ be a coherent $\OO_X$-module and $Y\subset \Spec A$ a
closed subscheme. Let us denote by  ${}^\wedge$ the completion
along  $Y$ (respectively, along $\pi^{-1}(Y)$). The theorem on
formal functions states that
$$H^i(X,\M)^\wedge = H^i(X,\hat\M)$$
Two important corollaries of this theorem are Stein's
factorization theorem and  Zariski's Main Theorem (\cite{Hart}
III, 11.4, 11.5).

Hartshorne \cite{Hart}  gives a proof of the theorem on formal
functions for projective schemes (over a ring). Grothendieck
\cite{Grot} proves it for proper schemes. He first gives
sufficient conditions for the commutation of the cohomology of
complexes of $A$-modules with inverse limits (0, 13.2.3
\cite{Grot}); secondly, he gives a general theorem on the
commutation of the cohomology of sheaves with inverse limits (0,
13.3.1 \cite{Grot}); finally, he laboriously checks that the
theorem on formal functions is under the hypothesis of this
general one (4.1.5 \cite{Grot}).


In this paper we give the ``obvious direct proof'' of the theorem
on formal functions. Very briefly, we prove that the completion of
the Godement resolution of a coherent sheaf is a flasque
resolution of the completion of the coherent sheaf and that taking
sections in the Godement complex  commutes with completion.

\section{Theorem on formal functions}

\begin{defn} Let $X$ be a scheme, $\pp\subset \OO_X$ a sheaf of ideals
and $\M$ an $\OO_X$-module. The  $\pp$-adic completion of $\M$,
denoted by $\wh\M$, is
$$\wh\M :=\plim{n}\M/\pp^n\M$$ \end{defn}

If $U=\Spec A$ is an affine open subset and $I=\pp(U)$, one has a
natural morphism
\[ \Gamma (U,\M)\otimes_A A/I^n \to  \Gamma (U,\M/\pp^n\M)\]
and then a morphism
\[ \Gamma (U, \M)^{\wedge}\to \Gamma (U,\wh\M)\]  where $\Gamma (U,
\M)^{\wh\quad}$ is the $I$-adic completion of $\Gamma (U,\M)$.

\begin{defn}  We say that $\M$ is affinely $\pp$-acyclic if for any
affine open subset  $U$ and any natural number $n$,  the sheaves
$\M$ and $\M/\pp^n\M$ are acyclic on $U$ and the morphism $\Gamma
(U,\M)\otimes_A A/I^n \to  \Gamma (U,\M/\pp^n\M)$ is an
isomorphism. In particular, $\Gamma (U, \M)^{\wedge}\to \Gamma
(U,\wh\M)$ is an isomorphism.\end{defn} Every quasi-coherent
module is affinely $\pp$-acyclic.

\noindent{\ \ \it{Notations}:}  For any sheaf $F$, let us denote
$$0\to F\to C^0F\to C^1F\to \cdots \to C^nF\to \cdots$$
its Godement resolution. We shall denote $C^\punto F=\oplusa{i\geq
0} C^i F$ and $F_i=\Ker (C^iF\to C^{i+1}F)$. One has that
$C^0F_i=C^iF$.

\begin{lem}\label{Affine} Let $X$ be a scheme, $\pp$
a coherent ideal and $\M$ an $\OO_X$-module. Denote
$I=\Gamma(X,\pp)$ and assume that $\pp$ is generated by a finite
number of global sections (this holds for example when $X$ is
affine). For any open subset $V\subseteq X$ one has
\[ \Gamma (V, C^0(\pp \M)) = I\cdot \Gamma (V,C^0\M)\]
In particular, the natural morphism $\pp C^0 M\to C^0(\pp \M)$ is
an isomorphism.\end{lem}

\begin{proof} If $J$ is a finitely generated ideal of a ring $A$ and
$M_i$ is a collection of $A$-modules, then $J\cdot \prod M_i =
\prod (J\cdot M_i)$. Now, by hypothesis $\pp$ is generated by a
finite number of global sections $f_1,\dots ,f_r$. Let
$J=(f_1,\dots,f_r)$. Then
\[ \Gamma (V, C^0(\pp \M)) =\prod_{x\in V} \pp_x\cdot \M_x = \prod_{x\in V} J\cdot \M_x
= J\cdot \prod_{x\in V} \M_x = J\cdot \Gamma (V,C^0\M)\] Since
$I\cdot \prod_{x\in V} \M_x $ is contained in $\Gamma (V, C^0(\pp
\M))$ one concludes. In particular, if $V$ is affine, then $
\Gamma (V, C^0(\pp \M)) = I_V\cdot \Gamma (V,C^0\M)$, with
$I_V=\Gamma (V,\pp)$. It follows that $\pp C^0 M\to C^0(\pp \M)$
is an isomorphism.
\end{proof}

\begin{prop} \label{CompFlas} Let $X$ be a scheme and let $\pp$ be a
coherent ideal. For any $\OO_X$-module $\M$ one has: \enumera
\item  $\pp  C^i\M=C^i(\pp\M)$ and
$(C^i\M)/\pp (C^i\M) = C^i(\M/\pp\M)$, for any $i$.

\item $C^0\M$ is affinely $\pp$-acyclic.

\item $\wh{C^0 \M }$ is flasque. Moreover, if $\pp$ is generated by a finite number of global sections, then
$$\Gamma(X, \wh{C^0\M})=\Gamma(X,C^0\M)^\wedge$$
\eenumera\end{prop}

\begin{proof} 1. We may assume that $X$ is affine.  Hence $\pp C^0 \M =
C^0 (\pp \M)$ by the previous lemma and $(C^0\M)/\pp C^0\M =
C^0\M/C^0(\pp\M) = C^0(\M/\pp\M)$. From the exact sequence
\[  \M/\pp \M  \to C^0\M/\pp C^0\M \to
\M_1/\pp\M_1 \to 0\] and the isomorphism $C^0\M/\pp C^0\M =
C^0(\M/\pp\M)$ it follows that $\M_1/\pp \M_1=(\M/\pp\M)_1$ and
$\pp\M_1=(\pp\M)_1$. Consequently $\pp C^1\M =\pp
C^0(\M_1)=C^0(\pp\M_1)=C^0((\pp\M)_1)=C^1(\pp\M)$, and analogously
$C^1\M/\pp C^1\M = C^1(\M/\pp\M)$. Repeating this argument one
concludes 1.

2. Denote $\Nc=C^0\M$. By (1), $\Nc/\pp^n\Nc$ is acyclic on any
open subset. From the long exact sequence of cohomology associated
to $0\to \pp^n\Nc\to \Nc\to \Nc/\pp^n\Nc\to 0$ and the acyclicity
of $\pp^n\Nc$ (by (1)) one obtains that  $$\Gamma (U,
\Nc/\pp^n\Nc) = \Gamma (U,\Nc)/\Gamma (U,\pp^n\Nc).$$ Moreover, if
$U$ is affine
 $\Gamma (U,\pp^n\Nc)= \pp^n(U)\Gamma (U,\Nc)$, by Lemma
 \ref{Affine}. We have concluded.

3. Let us prove that $\Nc=\wh{C^0 \M }$ is flasque. It suffices to
prove that its restriction to any affine open subset is flasque,
so we may assume that $X$ is affine. Let us denote $I=\pp(X)$. For
any open subset $V$, one has as in the proof of (2)
$$\Gamma (V,\wh{\Nc})= \plim{n}\Gamma (V, \Nc/\pp^n\Nc)  = \plim{n} \Gamma (V, \Nc)/\Gamma (V,
\pp^n\Nc)$$ and by Lemma \ref{Affine}, $\Gamma (V, \pp^n\Nc) =
I^n\Gamma (V, \Nc)$. In conclusion, $\Gamma (V,\wh{\Nc})= \Gamma
(V,\Nc)^{\wh\quad}$. One concludes that $\wh\Nc$ is flasque
because $\Nc$ is flasque and the $I$-adic completion preserves
surjections. The same arguments prove the second part of the
satement.
\end{proof}

\begin{prop} \label{CompGod} If $\M$ is affinely $\pp$-acyclic, then
$\wh{C^\punto\M}$ is a flasque resolution of $\wh\M$. \end{prop}

\begin{proof} We already know that $\wh{C^\punto\M}$ is flasque.
Let us prove now that $\M_1$ is affinely $\pp$-acyclic. From the
exact sequence
\[0\to \M/\pp^n \M  \to C^0(\M/\pp^n\M) \to \M_1/\pp^n\M_1 \to 0\]
one has that  $\M_1/\pp^n\M_1$ is acyclic on any affine open
subset. Moreover, taking sections on an affine open subset
$U=\Spec A$, one obtains the exact sequence (let us denote
$I=\pp(U)$)
\[ 0\to \Gamma (U,\M)\otimes_AA/I^n \to \Gamma (U,C^0 \M)\otimes_AA/I^n
 \to \Gamma (U,\M_1/\pp^n\M_1) \to 0\]
and then $\Gamma (U,\M_1)\otimes_AA/I^n = \Gamma
(U,\M_1/\pp^n\M_1)$, i. e. $\M_1$ is affinely $\pp$-acyclic.

Now, taking inverse limit in the above exact sequence  (and taking
into account that the $I$-adic completion preserves surjections)
one obtains the exact sequence
$$0\to \Gamma (U,\wh \M)\to \Gamma (U,\wh {C^0\M})\to \Gamma
(U,\wh {\M_1})\to 0$$ Therefore the sequence $0\to \wh\M\to
\wh{C^0\M} \to \wh{\M_1}\to 0$ is exact. Conclusion follows
easily. \end{proof}

\begin{rem} In the proof of the preceding proposition it has been proved
that if $\M$ is affinely $\pp$-acyclic, then $\wh\M$ is acyclic on
any affine subset.\end{rem}

\begin{lem} \label{SSL} Let $A$ be a noetherian ring and  $I\subseteq A$
an ideal. If $0\to M'\to M\to N\to 0$ is an exact sequence of
$A$-modules and  $N$ is finitely generated, then the $I$-adic
completion $0\to \wh {M'}\to \wh M\to \wh N\to 0$ is exact.
\end{lem}

\begin{proof} Let $L\subseteq M$ be a finite submodule surjecting on $N$
and $L'=L\cap M'$ which is also finite because $A$ is noetherian.
The exact sequences $$0\to L\to M\to M/L\to 0,\quad 0\to L'\to
M'\to M'/L'\to 0,\quad 0\to L'\to L\to N\to 0$$ remain exact after
$I$-adic completion, because $L$ and $L'$ are finite (this is a
consequence of Artin-Rees lemma (10.10 \cite{Atiy})). Since
$M/L\simeq M'/L'$ one concludes. \end{proof}




\begin{thm}[on formal functions ] Let $f\colon X\to Y$ be a proper
morphism of locally noetherian schemes,   $\pp$  a coherent sheaf
of ideals on $Y$ and $\pp\OO_X$ the ideal induced in $X$.  For any
coherent module $\M$ on $X$, the natural morphisms (where
completions are made by $\pp$ and $\pp\OO_X$ respectively)
$$\wh{R^if_*\M}\to R^if_*(\wh\M)$$ are isomorphisms. If $Y=\Spec
A$, then $$H^i(X,\M)^{\wedge}=H^i(X,\wh\M)$$ \end{thm}

\begin{proof} The question is local on $Y$, so we may assume that $Y=\Spec
A$ is affine.  It suffices to show that
$H^i(X,\M)^{\wh\quad}=H^i(X,\wh\M)$. It is clear that $\pp\OO_X$
is generated by its global sections. As usual, we denote $I=\Gamma
(X, \pp\OO_X)$.

Let $C^\punto \M$ be the Godement resolution of $\M$. Then
$\wh{C^\punto M}$ is a flasque resolution of $\wh{\M}$ (by
Proposition \ref{CompGod}) and $\Gamma (X, \wh{C^\punto \M})=
\Gamma (X, C^\punto\M)^{\wh\quad}$ (by Proposition \ref{CompFlas},
(3)). Then we have to prove that the natural map
$$H^i(X,\M)^{\wedge}= [H^i\Gamma(X,C^\punto\M)]^{\wedge}
\to H^i(\Gamma(X,C^\punto\M)^{\wedge})= H^i(\Gamma(X,\wh{C^\punto
\M})) = H^i(X,\wh{\M})$$ is an isomorphism. Let us denote by $d_i$
the differential of the complex $\Gamma (X,C^\punto\M)$ on degree
$i$. Completing the exact sequences
\[\aligned   0\to
\Ker d_i \to \,\,&\Gamma (X, C^{i}\M )\to \Ima d_i\to 0
\endaligned\] we obtain the exact sequences
\[\aligned
0\to \wh{\Ker d_i} \to\,\, &\Gamma (X, \wh{C^{i}\M})\to \wh{\Ima
d_i}\to 0
\endaligned\] because, as we shall see below, the $I$-adic
topology of $\Gamma (X,{C^{i}\M})$ induces in $\Ker d_i$ the
$I$-adic topology. Hence
$$ {H^i(X,\M)}^{\wedge} = (\Ker d_i/\Ima d_{i-1})^{\wedge}
\overset{\text{Lemma } \ref{SSL}} {=\hspace {-2pt} =\hspace {-2pt}
=\hspace {-2pt} =\hspace {-2pt} =} \wh{\Ker d_i}/\wh{\Ima
d_{i-1}}= H^i(X,\wh\M)
$$

Let $\M_i$ be the kernel of $C^{i}\M\to C^{i+1}\M$ (recall that
$C^{i}\M=C^0\M_i$). Let us prove that the $I$-adic topology of
$\Gamma (X, C^i \M)$ induces the $I$-adic topology on $\Ker
d_i=\Gamma(X,\M_i)$. Intersecting the equality $I^n \Gamma(X,
C^0\M_i) = \Gamma (X, C^0(\pp^n\M_i))$ with $\Gamma (X,\M_i )$,
one obtains that the induced topology on $\Gamma (X, \M_i )$ is
given by the filtration $\{ \Gamma (X, \pp^n\M_i)\}$. Hence it
suffices to show that this filtration is $I$-stable. Since
$\pp^n\M_i = (\pp^n\M)_i$ (see the proof of \ref{CompFlas} 1.), it
is enough to prove that the filtration $\{ \Gamma (X,
(\pp^n\M)_i)\}$ is $I$-stable; this is equivalent to show that
$\oplus^\infty_{n=0} \Gamma (X, (\pp^n\M)_i)$ is a $D_IA$-module
generated by a finite number of homogeneous components, where
$D_IA=\oplusa[\infty]{n=0} I^n$. By the exact sequence
$$\oplusa[\infty]{n=0} \Gamma (X, C^{i-1}(\pp^n\M)) \to
\oplusa[\infty]{n=0} \Gamma (X, (\pp^n\M)_i) \to
\oplusa[\infty]{n=0} H^i(X,\pp^n\M)\to 0$$ it suffices to see the
statement for the first and the third members. For the first one
is obvious  because $\Gamma (X, C^{i-1}(\pp^n\M))=I^n\Gamma (X,
C^{i-1}\M)$. For the third one, it suffices to see that it is a
finite $D_IA$-module. Let $X'=X\times_A D_IA$, $\pi\colon X'\to X$
the natural projection and $\M'=\oplusa[\infty]{n=0} \pp^n\M$ the
obvious $\OO_{X'}$-module. Since $H^i(X',\M')$ is a finite
$D_IA$-module, one concludes from the equalities
$H^i(X',\M')=H^i(X,\pi_*\M')= \oplusa[\infty]{n=0}
H^i(X,\pp^n\M)$, because $\pi_*\M'=\oplusa[\infty]{n=0} \pp^n\M$.
\end{proof}

\begin{rem} Reading carefully the above proof, it is not difficult to
see that one has already showed that  $H^i(X,\M)^\wedge = \plim{n}
H^i(X,\M/\pp^n\M)$. \end{rem}


\bibliographystyle{amsplain}
\bibliography{xbib}
\end{document}